\documentclass[11pt,a4paper]{amsart}
\usepackage{amsmath,amssymb,amsfonts, float}
\usepackage[all]{xy}
\usepackage{caption}
\usepackage{enumerate}
\usepackage{mathpazo}
\usepackage{a4wide}
\usepackage{tikz}
\usetikzlibrary{calc}
\usepackage{soul}

\usepackage{bm}
\usepackage{graphicx}
\usepackage[breaklinks=true]{hyperref}

\usepackage{xcolor}
\usepackage{multicol}
\usepackage{tabularx}

\usepackage{hyperref}
\usepackage[capitalize]{cleveref}

\usepackage[normalem]{ulem}

\newcommand{\jon}[1]{{\color{purple}#1}}

\newtheorem{thm}{Theorem}[section] 
\newtheorem*{thm*}{Theorem} 
\newtheorem{prop}[thm]{Proposition}
\newtheorem{lem}[thm]{Lemma}
\newtheorem{cor}[thm]{Corollary}

\theoremstyle{definition}
\newtheorem{definition}[thm]{Definition}
\newtheorem{expl}[thm]{Example}

\newtheorem{rem}[thm]{Remark}

\newtheorem{speculation}[thm]{Speculation}

\DeclareMathOperator{\R}{\mathbb{R}}
\DeclareMathOperator{\C}{\mathbb{C}}
\DeclareMathOperator{\Z}{\mathbb{Z}}

\DeclareMathOperator{\Div}{\text{Div}}

    \DeclareFontFamily{U}{wncy}{}
    \DeclareFontShape{U}{wncy}{m}{n}{<->wncyr10}{}
    \DeclareSymbolFont{mcy}{U}{wncy}{m}{n}
    \DeclareMathSymbol{\Sha}{\mathord}{mcy}{"58}

\numberwithin{equation}{section}

\renewcommand{\i}{\mathrm{i}}

\DeclareSymbolFont{bbold}{U}{bbold}{m}{n}
\DeclareSymbolFontAlphabet{\mathbbold}{bbold}

\usepackage{hyperref}

\newcommand{\A}{\mathbb{A}}
\newcommand{\ones}{J}

\newcommand{\s}{\text{ss}}

\thanks{JLM is supported by a research development account from Soka University of America. JM is partially supported by the Natural sciences
and Engineering Research Council of Canada (NSERC) grant Ro370A01. TTN is partially supported by an AMS-Simons Travel Grant and gratefully acknowledges the Mathematics and Computer Science Department at Lake Forest College for sponsoring the Overleaf subscription used for this project. NDT is partially supported by the Vietnam National Foundation for Science and Technology Development (NAFOSTED) under grant number 101.04-2023.21}

\begin{document}
\title{On Divisibility relation graphs}
\author[J. Merzel, J. Min\'a\v{c}, T.T. Nguyen, N.D. Tan]{Jonathan L. Merzel, J\'an Min\'a\v{c}, \\ Tung T. Nguyen, Nguy$\tilde{\text{\^{e}}}$n Duy T\^{a}n}

\address{Soka University of America}
\email{jmerzel@soka.edu }

\address{
University of Western Ontario}
\email{minac@uwo.ca}

\address{Lake Forest College}
\email{tnguyen@lakeforest.edu}

\address{Hanoi University of Science and Technology}
\email{tan.nguyenduy@hust.edu.vn}

\date{}

\maketitle
\begin{abstract}
For each positive integer $n$, we define the divisibility relation graph  $D_n$ whose vertex set is the set of divisors of $n$, and  in which two vertices are adjacent if one is a divisor of the other. This type of graph is a special case of graphs associated with a partial order, which have been widely studied in the literature. In this work, we determine various graph-theoretic invariants of divisibility relation graphs, such as their clique and independence numbers, and their planarity. We also discuss various spectral properties that are discovered by our numerical experiments. 

\end{abstract}

\section{Introduction}
The divisor graph $G_n$ on $n$ nodes is the graph whose vertices are $\{1, 2, \ldots, n \}$ and in which edges are $(a,b)$ where either $a \mid b$ or $b \mid a.$ Due to  their arithmetical nature, these graphs have been widely studied in the literature. For example, there exists an extensive line of work investigating the longest simple path on $G_n$ (see \cite{erdHos1995graphe, Melotti, pollington1983there, pomerance1983longest}). Additionally, the determination of the independence and clique numbers of $G_n$ has also attracted significant interest (see, for example, \cite{Rodrigo, Cameron1, McNew}).

While studying gcd-graphs in \cite{MTT5, nguyen2025gcd}, we naturally rediscovered $G_n$. Specifically, in our attempts to understand certain induced structures on these gcd-graphs and to prove several statements about their primeness, we realized the importance of understanding the connection scheme between the given divisor subset. Therefore, quite naturally, for our purposes, we restrict ourselves to the induced subgraph $D_n$ of $G_n$ on the set of divisors of $n$. We remark that the vertices of $D_n$ provide a complete set of representatives for the quotient space $(\mathbb{Z}/n\mathbb{Z})^{\times} \backslash (\mathbb{Z}/n\mathbb{Z})$, where the action of $(\mathbb{Z}/n\mathbb{Z})^{\times}$ on $(\mathbb{Z}/n\mathbb{Z} )$ is via the natural multiplication. Consequently, these vertices also form a complete system of representatives for certain supercharacters on $\mathbb{Z}/n\mathbb{Z}$ (see \cite{fowler2014ramanujan}). It is our hope that the study of divisibility relation graphs will offer further insights into the study of these supercharacters and their associated Ramanujan sums.

We note that the definition of $D_n$ is also introduced in \cite[Definition 19]{divisor_graph_1}.  To start our discussion, let us provide the following formal definition of $D_n.$

\begin{definition}
The divisibility relation graph $D_n$ is the graph with the following data 

\begin{enumerate}
    \item The vertex set of $D_n$ is the set of all divisors of $n.$
    \item Two vertices $a\neq b$ are adjacent if and only if $a \mid b$ or $b \mid a.$
\end{enumerate}
\end{definition}
The graph $D_n$ is a special case of compatibility graphs; i.e. graphs which are associated with an order (we will explain this interpretation in more detail in \cref{subsec:equivalent}). Therefore, $D_n$ is somewhat more structured than $G_n$. For example, a classical theorem of Dilworth shows that compatibility graphs are perfect (see \cite{Dilworth, Fulkerson}). In particular, $D_n$ is a perfect graph. Furthermore, we can explicitly calculate some of its variants, such as its clique number (and hence the chromatic number since $D_n$ is perfect), its independence number, and its planarity. More surprisingly, we found that the spectrum of $D_n$ has several interesting properties. Although some of them can be explained by number-theoretic or set-theoretic arguments, others are somewhat more mysterious. We now summarize our main results. We refer the reader to the main text for more precise statements.

\begin{thm}
For each positive integer $n$, we denote by $D_n$ the divisibility relation graph associated with $n.$ Let $f_n$ be the characteristic polynomial of $D_n.$ 

\begin{enumerate}
    \item $D_n$ is a planar graph if and only if $n \in \{1,p, p^2, p^3, pq, p^2 q\}$ for some distinct primes $p,q$. 
    \item Let $p,q$ be two distinct primes such that $\gcd(n,pq)=1.$ Then $f_{n}|f_{npq}.$ Furthermore, if there exists a prime $r$ such that $r \mid \mid n$, then $f_{n}^2 \mid f_{npq}.$
    \item $-1$ is always an eigenvalue of $D_n.$
    \item Suppose that $\mu(n)=-1$ and $\Omega(n) \geq 2$ where $\mu$ and $\Omega$ are respectively the Mobius and degree functions. Then $-2$ and $1$ are an eigenvalue of $D_n.$
    \item Suppose that $\mu(n)=1$. Then $0$ is an eigenvalue of $D_n.$
    \item Suppose that $n=pq^a$ where $p,q$ are distinct primes and $a \geq 1$. Then $0$ is an eigenvalue of $D_n$ if and only if $a \equiv 1 \pmod{6}.$
    \item Suppose that $n=p^a q^b$ where $p,q$ are distinct primes and $a,b \geq 1.$ If $a \equiv b \equiv 1 \pmod{6}$ then $0$ is an eigenvalue of $D_n.$
\end{enumerate}
\end{thm}

\subsection{Outline}
In \cref{sec:graph theoretic}, we discuss various graph-theoretic properties of $D_n$. Most arguments in this section are quite elementary, which makes it accessible to a wide range of readers. Our main discoveries are in \cref{sec:spectrum} where we discuss some variational properties of the spectrum of $D_n$ when $n$ varies. In particular, we find that the case where $n$ has at most two prime factors has some quite surprising patterns. 

\subsection{Code}
The code that we wrote to generate data and perform experiments with divisibility relation graphs can be found in \cite{nguyen_divisor_graph}.  We remark that we have also verified all statements in this work with various concrete and computable examples. 

\section{Graph theoretic properties of $D_n$} 
\label{sec:graph theoretic}

\subsection{Isomorphism classes and equivalent interpretations}
\label{subsec:equivalent}

For each natural number $n$, we can factor $n$ into a product of powers of distinct prime factors 
\[ n = p_1^{a_1} p_2^{a_2} \cdots p_d^{a_d}. \] 
If we require further that $1 \leq a_1 \leq a_2 \leq  \cdots \leq a_d$ and that $p_i<p_j$ when $i<j$ and $a_i=a_j$ then this factorization is unique. 
\begin{definition}
    We call $(a_1, a_2, \ldots, a_d)$ the factorization type of $n.$ We will write 
    \[ F(n)= (a_1, a_2, \ldots, a_d). \]
Following the convention in \cite{divisor_numbers}, we also define the degree of $n$ to be $\Omega(n)= \sum_{i=1}^d a_i.$
\end{definition}
\begin{rem}
In this article, the term \textit{degree} could refer to two distinct notions: the degree of a natural number $n$, denoted by $\Omega(n)$, and the degree of $n$ considered as a vertex in $D_n$, denoted by $\deg(n)$. It should be clear from the context which one is being referenced.

\end{rem}

We now describe an equivalent definition of $D_n.$ Let $S$ be the following set 
\[ S =\{ (\alpha_1, \alpha_2, \ldots, \alpha_d) \in \Z^d \mid 0 \leq \alpha_i \leq a_i \quad \forall 1 \leq i \leq d \}.\]

\begin{definition}
Let $ \bm{\alpha} = (\alpha_1, \alpha_2, \ldots, \alpha_d)$ and $ \bm{\beta} = (\beta_1, \beta, \ldots, \beta_d)$ be two vectors in $\R^{d}.$ We say that $\bm{\alpha} \leq \bm{\beta}$ if and only if $\alpha_i \leq \beta_i$ for all $1 \leq i \leq d.$ 
\end{definition}
Restricting this partial order on $\R^{d}$ to $S$, we get a partially ordered set $(S,\le)$.  Let $\Div(n)$ be the set of divisors of $n.$ For each $d \in \Div(n)$, $d$ has a unique factorization of the form $\prod_{i=1}^d p_i^{\alpha_i}$ where $0 \leq \alpha_i \leq a_i.$ As a result, we have a canonical a map $\Phi\colon \Div(n) \to S$  defined by $\Phi(d)=(\alpha_1, \ldots, \alpha_d).$ By the fundamental theorem of arithmetic, $\Phi$ is a bijection. Furthermore, $d_1 | d_2$ if and only if $\Phi(d_1) \leq \Phi(d_2).$ We conclude that $D_n$ is ismorphic to the compatibility graph associated with $(S, \leq).$ In particular, the isomorphism class of $D_n$ only depends on the factorization type of $n.$  While this interpretation of $D_n$ may seem almost tautological, it proves useful—especially when constructing the graph using the Python library networkx, as it allows us to avoid explicit factorization of large numbers.

\subsubsection{Equivalent description via lattice theory.}
We can also interpret $D_n$ as a lattice of intermediate subfields of the finite field $\mathbb{F}_{p^n}$ where $p$ is a prime number. Indeed it is a fundamental basic result in theory of finite fields that there is a 1-1 correspondence (see \cite[Theorem 2.6]{lidl1997finite}) 
$$ \{ d \mid \text{such that } d \mid n \} \longleftrightarrow \{\text{all fields } \mathbb{F} \subset \mathbb{F}_{q}: = \mathbb{F}_{p^n} \} $$
$$ d \longrightarrow \mathbb{F}_{p^d} $$

We see therefore that each vertex $d \in D_n$ correspond exactly one subfield $\mathbb{F}_{p^d}$ of $\mathbb{F}_{p^n}$. Furthermore, $a \mid b$ if and only if $\mathbb{F}_{p^a} \subset \mathbb{F}_{p^b}$. Therefore we attain an isomorphism
$$ D_n \cong L(\mathbb{F}_q) := \text{Lattice of intermediate subfields of } \mathbb{F}_q \text{ viewed as a compatibility graph}.$$
\begin{expl}
Let us discuss a concrete example. The lattice of all subfields of $\mathbb{F}_{2^{36}}$ is described in \cref{fig:subfield_lattice}. We remark that we only draw edges between fields $K$ and $L$ where there is no proper subfield $N$, $K \subset N \subset L$ and $K \subset L$. 

Considering the graph structure of $D_{36} \cong L(\mathbb{F}{2^{36}})$, we have the following remark about the degree distribution: the vertices representing the trivial subfield $\mathbb{F}_2$ and the full field $\mathbb{F}_{2^{36}}$ both have degree $8$.  The vertices for $\mathbb{F}_{2^4}$ and $\mathbb{F}_{2^{9}}$ have degree $4$. Finally, the vertices corresponding to $\mathbb{F}_{2^{2}}$, $\mathbb{F}_{2^{3}}$, $\mathbb{F}_{2^{6}}$, $\mathbb{F}_{2^{12}}$, and $\mathbb{F}_{2^{18}}$ each have degree $6$.
\begin{figure}
\label{tikz:fields}
\begin{tikzpicture}[scale=0.8, every node/.style={circle, draw, inner sep=2pt}]

  \node (F1) at (0,0) {$\mathbb{F}_2$};       
  \node (F2) at (-2,2) {$\mathbb{F}_{2^2}$};    
  \node (F3) at (2,2) {$\mathbb{F}_{2^3}$};    
  \node (F4) at (-4,4) {$\mathbb{F}_{2^4}$};    
  \node (F6) at (0,4) {$\mathbb{F}_{2^6}$};     
  \node (F9) at (4,4) {$\mathbb{F}_{2^9}$};     
  \node (F12) at (-2,6) {$\mathbb{F}_{2^{12}}$};   
  \node (F18) at (2,6) {$\mathbb{F}_{2^{18}}$};   
  \node (F36) at (0,8) {$\mathbb{F}_{2^{36}}$};  

  \draw (F1) -- (F2);   
  \draw (F1) -- (F3);   

  \draw (F2) -- (F4);   
  \draw (F2) -- (F6);   

  \draw (F3) -- (F6);   
  \draw (F3) -- (F9);   

  \draw (F4) -- (F12);  

  \draw (F6) -- (F12);  
  \draw (F6) -- (F18);  

  \draw (F9) -- (F18);  

  \draw (F12) -- (F36); 
  \draw (F18) -- (F36); 

\end{tikzpicture}
\caption{Lattice of subfields of $\mathbb{F}_{2^{36}}$.}
\label{fig:subfield_lattice}  
\end{figure}
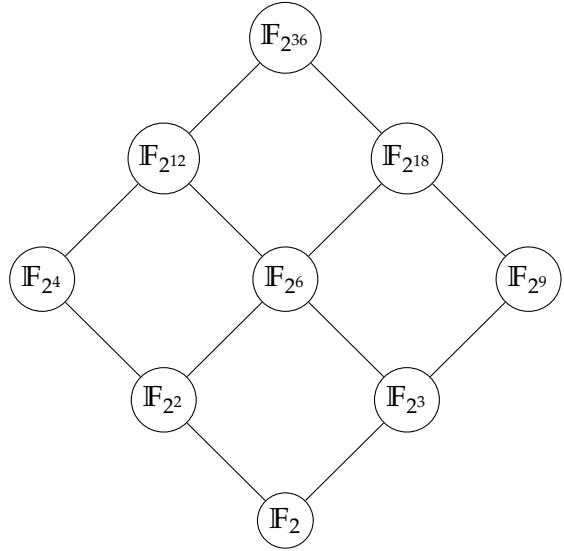

\end{expl}
\bigskip

\subsubsection{The case $n$ is squarefree.}

In the case where $n$ is squarefree, we can give an explicit description of
the adjacency matrix for $D_{n}$. \ 

\begin{prop}
Let $n$ be a squarefree positive integer, written as a product $%
p_{0}p_{1}\cdots p_{k-1}$ of distinct primes. \ Index the rows and columns
of the $2^{k}\times 2^{k}$ adjacency matrix by $\{0,1,\cdots ,2^{k}-1\}$. \
Then, using an appropriate ordering of vertices,  for $0\leq i,j\leq 2^{k}-1
$ the $i,j$ entry of the adjacency matrix for $D_{n}$ is $\binom{i}{j}
+\binom{j}{i}$ mod 2. 
\end{prop}

\begin{proof}
Order the vertices to correspond to row and column indices as follows: for $%
0\leq i\leq 2^{k}-1$, find the base $2$ representation $i=\sum%
\limits_{r=0}^{k-1}\varepsilon^{(i)} _{r}2^{r}$ ($\varepsilon^{(i)} _{r}\in \{0,1\}$)
and associate to this index the vertex  $\prod\limits_{r=0}^{k-1}{p_r}^{%
\varepsilon^{(i)} _{r}}$. \  Lucas' theorem with modulus 2 then states that %
$\binom {i}{j} \equiv \prod\limits_{r=0}^{k-1} \binom{\varepsilon^{(i)}_r}{\varepsilon^{(j)}_r} \mod2.$  Observing that$\binom{0}{0}=\binom{1}{0}=\binom{1}{1}$ and that$\binom{0}{1}=0$, we see that $\binom{i}{j}=1 \iff \varepsilon^{(i)}_r \ge \varepsilon^{(j)}_r$  for all $r$, and the result easily follows.
\end{proof}

\bigskip




\subsection{Degrees and edges in $D_n$}

\begin{prop}
If the factorization type of $n$ is $(a_{1},\ldots ,a_{d})$ then the graph $%
D_{n}$ has $v:=$ $\prod\limits_{i=1}^{d}(a_{i}+1)$ vertices and $e:=v\left( 
\frac{(a_{1}+2)\cdots (a_{d}+2)}{2^{d}}-1\right) $ edges. 
\end{prop}

\begin{proof}
The vertex count is just the standard result on the number of divisors of an
integer. \ To count edges, we orient the edges by considering an edge
joining vertices $b,c$ to originate at $b$ and terminate at $c$ if $\ b|c.$ 
 \ Then, if for $i=1$,$\ldots ,d$ we have $0\leq x_{i}\leq a_{i}$, the
number of edges originating at $b=\prod\limits_{i=1}^{d}p_{i}^{x_{i}}$ is $%
\prod\limits_{i=1}^{d}(a_{i}+1-x_{i})-1$. \ (We subtract 1 as we do not
include a loop at $b$.) \ So the total number of edges is given by 
\[
\sum\limits_{x_{1}=0}^{a_{1}}\cdots \sum\limits_{x_{d}=0}^{a_{d}}\left(
\prod\limits_{i=1}^{d}(a_{i}+1-x_{i})-1\right) 
\]%
Reindexing by setting $w_{i}=a_{i}+1-x_{i}$, we can rewrite this as%
\begin{eqnarray*}
\sum\limits_{w_{1}=1}^{a_{1}+1}\cdots \sum\limits_{w_{d}=1}^{a_{d}+1}\left(
\prod\limits_{i=1}^{d}w_{i}-1\right) 
&=&\sum\limits_{w_{1}=1}^{a_{1}+1}\cdots
\sum\limits_{w_{d}=1}^{a_{d}+1}\prod\limits_{i=1}^{d}w_{i}-v \\
&=&\left( \sum\limits_{w_{1}=1}^{a_{1}+1}w_{1}\right) \cdots \left(
\sum\limits_{w_{d}=1}^{a_{d}+1}w_{d}\right) -v \\
&=&\frac{(a_{1}+1)(a_{1}+2)}{2}\cdots \frac{(a_{d}+1)(a_{d}+2)}{2}-v= \\
&=&v\left( \frac{(a_{1}+2)\cdots (a_{d}+2)}{2^{d}}-1\right) 
\end{eqnarray*} \\
Alternatively, we could argue as  follows. To count edges, add a loop to each vertex and orient the
edges by viewing an edge connecting $a$ to $b$, where $a|b$, to originate at 
$a$ and terminate at $b$. \ Then the number of edges terminating at $b$ is
the number $d(b)$ of divisors of $b$. \ Thus, the total number of edges in $%
D_{n}$, plus a loop at each vertex, is%
\[
f(n):=\sum\limits_{b|n}d(b)
\]%
which is a multiplicative function. \ Evaluating at a prime power,$$%
f(p^{a})=\sum\limits_{i=0}^{a}d(p^{i})=\sum\limits_{i=0}^{a}(i+1)=\binom{a+2}{2},$$so for $n=p_{1}^{a_{1}}\cdots p_{d}^{a_{d}}$ we have $%
f(n)=\prod\limits_{i=1}^{d} 
\binom{a+2}{2}%
 $, and, subtracting the loops, the number of edges in $D_{n}$ is $%
e:=\prod\limits_{i=1}^{d}\binom{a+2}{2} -v =v\left( \frac{(a_{1}+2)\cdots (a_{d}+2)}{2^{d}}-1\right) .$ 
\end{proof}

\subsubsection{Vertices of Minimal Degree}
\begin{definition}
\ Let $\mathbf e_{i}$ be the unit vector with $0$ in position $j$, $%
1\leq j\ (\neq i)\leq d$, and $1$ in position $i$. \ Vectors $\mathbf{%
x}=(x_{1},\ldots ,x_{d})$ and $\mathbf{x}\pm \mathbf e_{i}$
for some $i$ will be called \underline{close}. \ A sequence $\mathbf{%
x}_{1},\ldots ,\mathbf{x}_{n}$ is a \underline{chain} iff $%
\mathbf{x}_{i}$ and $\mathbf{x}_{i+1}$ are close, $1\leq
i\leq n-1$. 
\end{definition}

\begin{definition}
A vector $\mathbf{x}=(x_{1},\ldots ,x_{d})$ is called \underline{%
extremal in coordinate $i$} if $x_{i}=0$ or $x_{i}=a_{i}$, and is called 
\underline{extremal} if it is extremal in all coordinate positions. The term
nonextremal applies otherwise in both situations.
\end{definition}

\begin{rem}
The degree of a vertex $\mathbf{x}=(x_{1},\ldots ,x_{d})$ is easily
seen to be given by $$f(\mathbf{x}):=\prod%
\limits_{i=1}^{d}(a_{i}+1-x_{i})+\prod\limits_{i=1}^{d}(x_{i}+1)-2.$$ \ For $%
1\leq i\leq d$ let $\Delta _{i}(\mathbf{x})=f(\mathbf{x}+%
\mathbf e_{i})-f(\mathbf{x})=\prod\limits_{j\neq
i}^{{}}(x_{j}+1)-\prod\limits_{j\neq i}^{{}}(a_{j}+1-x_{j})$, and note that $%
\Delta _{i}(\mathbf{x})$ does not depend on the value of $x_{i}$. \
It follows that if $\Delta _{i}(\mathbf{x})<0$, 
\[
\mathbf{x}-x_{i}\mathbf e_{i},\mathbf{x}-(x_{i}-1)%
\mathbf e_{i},\ldots ,\mathbf{x}+(a_{i}-x_{i})%
\mathbf e_{i}
\]%
is a chain through $\mathbf{x}$, joining the vectors $(x_{1},\ldots
,x_{i-1},0,x_{i+1},\ldots ,x_{d})$ and \\ $(x_{1},\ldots
,x_{i-1},a_i,x_{i+1},\ldots ,x_{d})$ extremal in coordinate $i$, with strictly
decreasing degrees. \ Similarly if $\Delta _{i}(\mathbf{x})>0$, we
get such a chain with strictly increasing degrees, and if $\Delta _{i}(%
\mathbf{x})=0$, the chain maintains constant degree. 
\end{rem}

\begin{prop}
A vertex of minimal degree can always be found among the extremal vertices.
\ In fact, any vertex  $\mathbf{x}$ of minimal degree satisfies the "stability condition" $%
\Delta _{i}(\mathbf{x})=0$ at every nonextremal coordinate, and
admits a chain from  $\mathbf{x}$ to an extremal vertex\ of minimal
degree where every member of this chain has minimal degree.
\end{prop}

\begin{proof}
The argument
in the above remark implies that $\mathbf{x}$, if not alreadly
extremal, must satisfy $\Delta _{i}(\mathbf{x})=0$ at every
nonextremal coordinate (else we can move to a close vertex of lower degree); choosing a nonextremal coordinate position $i$, $\mathbf{x}$ occurs in
a chain with constant degree leading to a vertex differing
only in coordinate $i$ where it is now extremal. \ Since all vertices in this chain are still of minimal degree, the stability condition is still intact. \
Repeat\ with any remaining nonextremal coordinate positions until arriving
at an extremal vertex. \ The construction ensures that the degree is not
changed in this process.
\end{proof}

\begin{prop}
An extremal vertex $\mathbf{x}$ has minimal degree is and only if $%
\prod\limits_{i\in A}(a_{i}+1)+\prod\limits_{j\in B}(a_{j}+1)$ is as small
as possible among extremal vertices, where $A=\{i~|~x_{i}=a_{i}\}$ and $%
B=\{j~|~x_{j}=0\}.$
\end{prop}

\begin{proof}
This follows simply from the formula for $f(\mathbf{x})$, which for
an extremal vertex $\mathbf{x}$ computes to $\prod\limits_{i\in
S}(a_{i}+1)+\prod\limits_{j\in T}(a_{j}+1)-2$.
\end{proof}

\subsection{Connectedness and bipartite property.}
The graph $D_n$ is always connected since $1$ and $n$ are adjacent to all other vertices. Furthermore, we have the following property. 
\begin{prop}
    Let $a,b$ be two vertices in $D_n.$ Let $d(a,b)$ be the distance between $a$ and $b.$ Then 
    \[
d(a,b) = 
\begin{cases} 
1 & \text{if } (a,b) \in E(D_n) \\
2 & \text{otherwise.} 
\end{cases}
\]
\end{prop}

\begin{proof}
    If $a$ and $b$ are adjacent then $d(a,b)=1.$ Otherwise, we have the path $a \to 1 \to b$ and hence $d(a,b)=2.$
\end{proof}

\begin{prop}
    Let $G$ be the induced graph on $D_n \setminus \{1, n \}.$ Then, $G$ is connected unless $n=pq$ where $p,q$ are distinct prime numbers; i.e, the factorization type of $n$ is $(1,1).$ 
\end{prop}

\begin{proof}
    If $n=pq$ where $p,q$ are distinct primes then we can see that the induced subgraph on $\{p,q \}$ is the empty graph--that is, a graph with vertices but no edges--and hence it is not connected. Conversely, suppose that the factorization type of $n$ is not $(1,1).$ If $n$ has exactly one prime divisor then the induced subgraph on $D_{n} \setminus \{1, n \}$ is a complete graph and therefore it is connected. Let us now assume that $n$ has at least two distinct prime factors. Since a vertex is adjacent to all its divisors, it is sufficient to show that if $p,q$ are distinct prime divisors of $n$ then $p,q$ are connected. This is clear since we have the path $p \to pq \to q.$
\end{proof}

\begin{prop}
 $D_n$ is bipartite if and only $F(n)=(1)$; i.e, $n$ is a prime number. 
\end{prop}
\begin{proof}
If $n$ is not a prime number then there exists $1<d<n$ such that $d \mid n.$ In this case $\{1, d, n \}$ forms a triangle in $D_n.$ As a result, $D_n$ cannot be bipartite. 
\end{proof}


\subsection{Clique and independent numbers of $D_n$}

\begin{prop}
    Suppose that $F(n)=(a_1, a_2, \ldots, a_d).$ Let $X$ be a maximal clique in $D_n$. Then
    \[ |X| = 1+ \sum_{i=1}^d a_i = 1+\Omega(n). \] 
    \end{prop}

\begin{proof}
We will prove by induction on $d$ that there is a clique of size $1+\Omega(n).$ For $d=1$, $n=p_1^{a_1}$ is a prime power and hence $D_n$ is the complete graph on $a_1+1$ nodes. In this case, $X=V(D_n)$ is a clique of size $1+ \Omega(n).$ Let us assume that the statement is true for $d$. Let us show that it is true for $d+1$ as well. Let us write $n=\prod_{i=1}^{d+1} p_i^{a_i}.$ Let $X=\{h_1, h_2, \ldots, h_{1+\sum_{i=1}^d a_d} \}$ be a clique of size $1+\sum_{i=1}^d a_i$ in $D_{\prod_{i=1}^d p_i^{a_i}}$. We can then check that $X' = \{1, p_{d+1}, \ldots, p_{d+1}^{a_d-1} \} \cup p_{d+1}^{a_d} X$ is a clique of size $1+\sum_{i=1}^{d+1} a_i$ in $D_{\prod_{i=1}^{d+1} p_i^{a_i}}.$

Conversely, we show that a clique in $D_n$ has at most $1 + \Omega(n)$ elements, In fact, let $X=\{h_1, h_2, \ldots, h_q\}$ be a clique, where $h_1 < h_2 <\ldots<h_q.$ We claim that $|X| \leq 1 + \Omega(n).$ Since $X$ is a clique and $h_i<h_{i+1}$, $h_i | h_{i+1}$ for each $1 \leq i \leq q-1$. This shows that $0 \leq \Omega(h_1) < \Omega(h_2)< \ldots < \Omega(h_q) \leq \Omega(n).$ Consequently, we conclude that  $|X| = q \leq 1 +\Omega(n).$
\end{proof}

For the independence number of $D_n$, we have the following result. 
\begin{prop} (See \cite[Theorem 1]{divisor_numbers})
Let $S$ be a subset of $D_n$ consisting of vertices of the same degree $ \lfloor \frac{\Omega(n)}{2} \rfloor$. Then $S$ is a maximal independence set in $D_n.$ In particular, the independence number of $D_n$ is precisely $|S|. $
\end{prop}

\subsubsection{Perfectness and coloring of $D_n$}



Since $D_n$ is perfect, the chromatic number of $D_n$ is equal to its clique number; i.e, $\chi(D_n) = 1 + \Omega(n).$ In fact, we can provide an explicit coloring on $D_n$ with $1+ \Omega(n)$ colors. Specifically, we can define a coloring map $c: D_n \to \{0, 1, \ldots, \Omega(n) \}$ by the rule $c(m)=\Omega(m).$ Clearly, two numbers with the same degree cannot be adjacent, so $c$ is a coloring map.


    

\subsection{Planarity of $D_n$}
In this section, we classify  all $D_n$ which are planar. By exploring $D_n$ with some small values of $n$, we can see that $F(n)=(a_1, a_2, \ldots, a_d)$ with $d \leq 2$ and $\sum_{i=1}^d a_i \leq 3$ then $D_n$ is planar. See \cref{fig:D45} for a planar embedding when $F(n)=(1,2)$ and \cref{fig:D27} for a planar embedding when $F(n)=(3).$

\begin{figure}[!htb]
\parbox{6cm}{
\includegraphics[scale = 0.4]{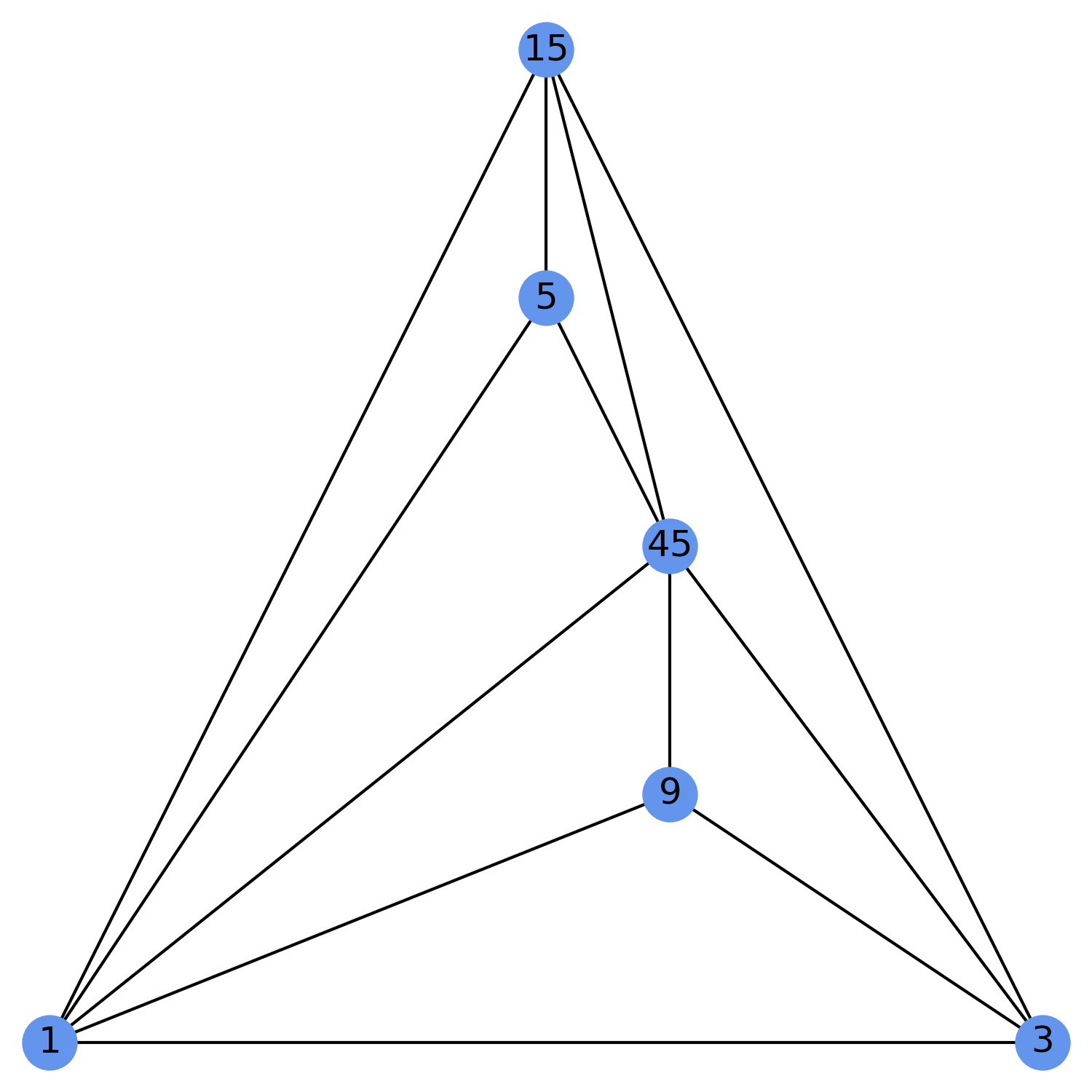}
\caption{A planar embedding of $D_{45}$}
\label{fig:D45}}
\qquad \quad
\begin{minipage}{6cm}
\includegraphics[scale = 0.4]{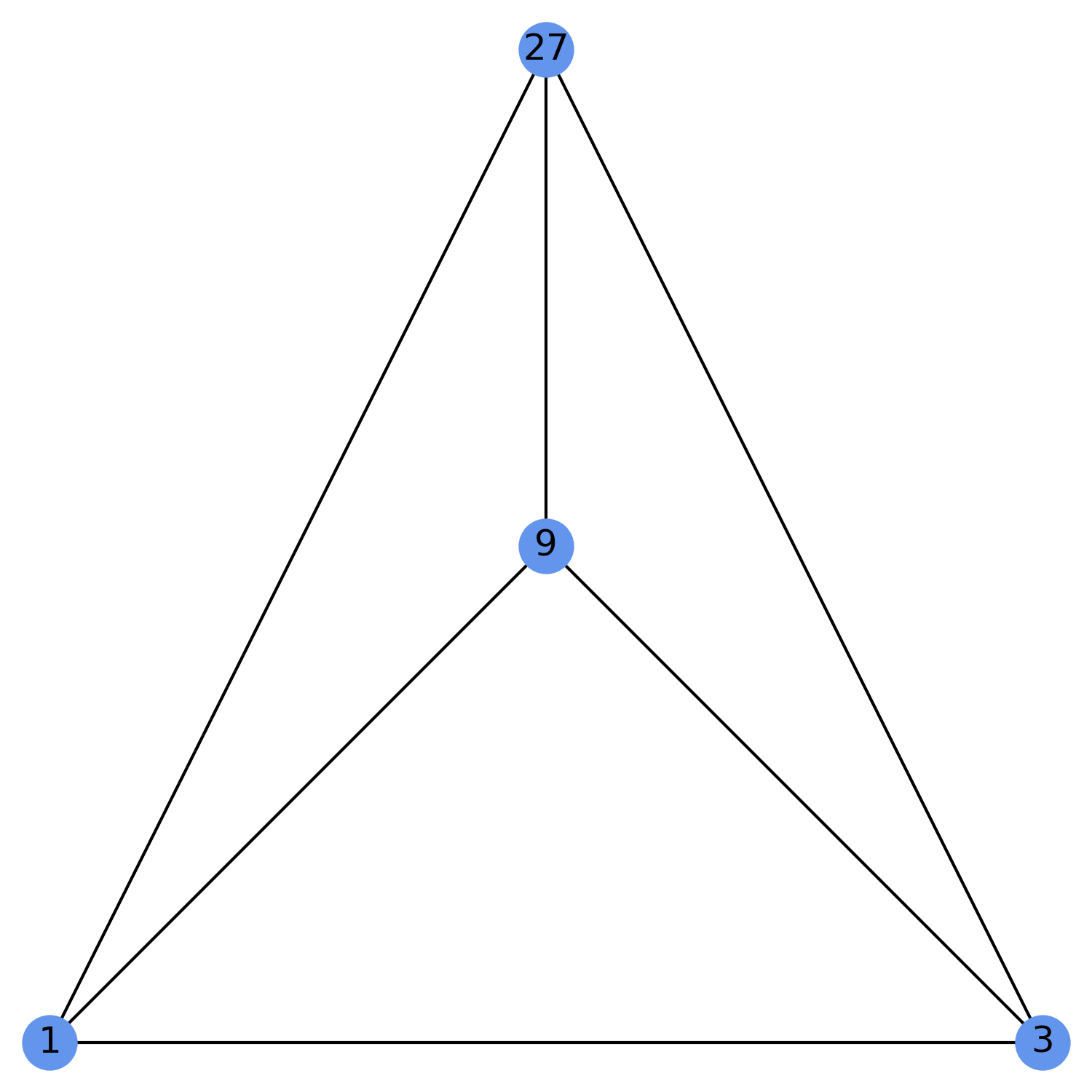}
\caption{A planar embedding of $D_{27}$.}
\label{fig:D27}
\end{minipage}
\end{figure}

The following theorem completely classifies $n$ such that $D_n$ is planar. 

\begin{prop}
The graph $D_{n}$ is planar if and only if one of the following holds:
\begin{enumerate}
\item $n = p^{r},\ p$ prime, $0 \leq r \leq 3$ 
\item $n = pq,\ p,q$ prime 
\item $n = p^{2}q,\ p,q$ prime
\end{enumerate}
\end{prop}

\begin{proof}
We make use of the fact that a graph cannot be planar if it contains a
subgraph isomorphic to $K_{3,3}$ or to a subdivision of $K_{5}$.
Since if $D_{n}$ is planar and $m|n$, $D_{m}$ (a subgraph)
is also planar, and if $D_{n}$ is not planar and $n|m$, $D_{m}$ is also not
planar, it suffices to check the following cases (where $p,q,r$ are distinct
primes) 

\begin{enumerate}
\item \ $D_{p^{3}}$ is planar (see \cref{fig:D27}) and $D_{p^{4}}$ is not. 
This is true since $D_{p^{r}}$ is the complete graph $K_{r+1}$ which is
planar iff $r+1 \leq 4$.
\item \ $D_{pq}$ and $D_{p^{2}q}$ are planar (see \cref{fig:D45}). 
\item \ $D_{p^{2}q^{2}}$ is not planar.
For it contains a subgraph which is a subdivision of $K_{5}$.
\begin{center}
\begin{tikzpicture}

\coordinate (A) at (90:3); 
\coordinate (B) at (162:3); 
\coordinate (C) at (234:3); 
\coordinate (D) at (306:3); 
\coordinate (E) at (18:3); 

\draw[thick] (A) -- (B) -- (C) -- (D) -- (E) -- cycle;

\draw[thick] (A) -- (C);
\draw[thick] (A) -- (D);
\draw[thick] (B) -- (D);
\draw[thick] (B) -- (E);
\draw[thick] (C) -- (E);

\node[above] at (A) {$1$};
\node[left] at (B) {$p^2q$};
\node[below left] at (C) {$p^2q^2$};
\node[below right] at (D) {$q$};
\node[right] at (E) {$p$};

\coordinate (F) at ($(C)!0.33!(D)$); 
\coordinate (G) at ($(C)!0.66!(D)$);
\coordinate (H) at ($(E)!0.5!(D)$); 

\node[right] at (H) {$pq$};

\foreach \point in {A, B, C, D, E, H} {
    \draw[blue, thick] (\point) circle (0.05);
}
\end{tikzpicture}

\end{center}

\item  \ $D_{p^{3}q}$ is not planar.
It contains the following subgraph which isomorphic to $K_{3,3}$. 

\begin{center}
\begin{tikzpicture}
    \coordinate (A1) at (0, 2);
    \coordinate (A2) at (0, 0);
    \coordinate (A3) at (0, -2);

    \coordinate (B1) at (2, 2);
    \coordinate (B2) at (2, 0);
    \coordinate (B3) at (2, -2);

    \foreach \i in {1,2,3} {
        \foreach \j in {1,2,3} {
            \draw[thick] (A\i) -- (B\j);
        }
    }

    \foreach \point in {A1, A2, A3, B1, B2, B3} {
        \draw[blue, thick] (\point) circle (0.05);
    }

    \node[above] at (A1) {$1$};
    \node[below] at (A2) {$p$};
    \node[below] at (A3) {$q$};
    \node[above] at (B1) {$pq$};
    \node[below] at (B2) {$p^2q$};
    \node[below] at (B3) {$p^3q$};
\end{tikzpicture}

\end{center}

\item \ $D_{pqr}$ is not planar. 

\begin{center}
\begin{tikzpicture}

\coordinate (A) at (90:3); 
\coordinate (B) at (162:3); 
\coordinate (C) at (234:3); 
\coordinate (D) at (306:3); 
\coordinate (E) at (18:3); 

\draw[thick] (A) -- (B) -- (C) -- (D) -- (E) -- cycle;

\draw[thick] (A) -- (C);
\draw[thick] (A) -- (D);
\draw[thick] (B) -- (D);
\draw[thick] (B) -- (E);
\draw[thick] (C) -- (E);

\node[above] at (A) {$1$};
\node[left] at (B) {$pqr$};
\node[below left] at (C) {$pq$};
\node[below right] at (D) {$r$};
\node[right] at (E) {$p$};

\coordinate (F) at ($(C)!0.33!(D)$); 
\coordinate (G) at ($(C)!0.66!(D)$); 

\coordinate (H) at ($(E)!0.5!(D)$); 

\node[below] at (F) {$q$};
\node[below ] at (G) {$rq$};
\node[right] at (H) {$pr$};

\foreach \point in {A, B, C, D, E, F, G, H} {
    \draw[blue, thick] (\point) circle (0.05);
}
\end{tikzpicture}

\end{center}

It fails the known necessary planarity condition $e \leq 3v - 6$. \
Alternatively, we note that it contains a subgraph which is a subdivision of
$K_{5}$.
\end{enumerate}
\end{proof}





\section{Spectrum of $D_n$} \label{sec:spectrum}
In this section, we study the spectrum of $D_n$. We focus on two main themes: 
\begin{enumerate}
\item How the spectra change when adding more prime factors to $n$, 
\item Some special eigenvalues of $D_n.$

\end{enumerate}

\subsection{Variation of the spectra of $D_n$}
For each $n$, let $f_n$ be the characteristic polynomial of $D_n.$
\begin{thm} \label{thm:main}
    Let $n$ be a natural number. Let $p,q$ be distinct prime numbers such that $\gcd(n,pq)=1.$ Then $f_{n} \mid f_{npq}.$
\end{thm}

\begin{proof}
Suppose that $p,q$ are distinct primes and that $\gcd (n,pq)=1$. \ Let 
$1=a_{1}<\cdots <a_{k}=n$ be the divisors of $n$, and let $B$ be the
adjacency matrix of $D_{n}$ with respect to this order. \ Let the $k\times k$
matrix $C$ be defined by $C_{ij}=1$ if $a_{i}|a_{j}$ and $0$ otherwise. \
(Then $B=C+C^{T}-2I_{k}$.) \ Now list divisors of $npq$ in the order $%
a_{1},\cdots ,a_{k},pa_{1},\cdots ,pa_{k},qa_{1},\cdots
,qa_{k},pqa_{1},\cdots ,pqa_{k}$. \ With respect to this order, the
adjacency matrix for $D_{npq}$ is%
\[
M=\left[ 
\begin{array}{cccc}
B & C & C & C \\ 
C^{T} & B & 0 & C \\ 
C^{T} & 0 & B & C \\ 
C^{T} & C^{T} & C^{T} & B%
\end{array}%
\right] 
\]%
Write $I_{k}$ for the $k\times k$ identity matrix and $\mathbf{0}_{k}$ for
an $k\times k$ matrix of $0$'s. \ The characteristic polynomial of $M$ can
be analyzed as follows, where for the third equality, we have subtracted the
second row of blocks from the third:%
\begin{eqnarray*}
f_{npq} &=&\det (M-\lambda I_{n+4})=\det\left[ 
\begin{array}{cccc}
B-\lambda I_{k} & C & C & C \\ 
C^{T} & B-\lambda I_{k} & \mathbf{0}_{k} & C \\ 
C^{T} & \mathbf{0}_{k} & B-\lambda I_{k} & C \\ 
C^{T} & C^{T} & C^{T} & B-\lambda I_{k}%
\end{array}%
\right]  \\
&=&\det \left[ 
\begin{array}{cccc}
B-\lambda I_{k} & C & C & C \\ 
C^{T} & B-\lambda I_{k} & \mathbf{0}_{k} & C \\ 
\mathbf{0}_{k} & \lambda I_{k}-B & B-\lambda I_{k} & \mathbf{0}_{k} \\ 
C^{T} & C^{T} & C^{T} & B-\lambda I_{k}%
\end{array}%
\right]  \\
&=&\det \left[ 
\begin{array}{cccc}
I_{k} & \mathbf{0}_{k} & \mathbf{0}_{k} & \mathbf{0}_{k} \\ 
\mathbf{0}_{k} & I_{k} & \mathbf{0}_{k} & \mathbf{0}_{k} \\ 
\mathbf{0}_{k} & \mathbf{0}_{k} & B-\lambda I_{k} & \mathbf{0}_{k} \\ 
\mathbf{0}_{k} & \mathbf{0}_{k} & \mathbf{0}_{k} & I_{k}%
\end{array}%
\right] \det \left[ 
\begin{array}{cccc}
B-\lambda I_{k} & C & C & C \\ 
C^{T} & B-\lambda I_{k} & 0 & C \\ 
\mathbf{0}_{k} & -I_{k} & I_{k} & \mathbf{0}_{k} \\ 
C^{T} & C^{T} & C^{T} & B-\lambda I_{k}%
\end{array}%
\right]  \\
&=&\det (B-\lambda I_{k})\cdot \text{(a polynomial in }\lambda \text{)}%
=f_{n}\cdot \text{(a polynomial in }\lambda \text{).}
\end{eqnarray*}%
This shows $f_{n}~|~f_{npq}$.
\end{proof}
By similar arguments, we can show:
\begin{thm} \label{thm:main2}
   Let $n$ be a natural number, $p,q$ distinct primes not dividing $n$. If there is a prime number $r$ such that $r\mid \mid n$, then $f_n^2|f_{npq}$. In particular, this is the case when $n>1$ is squarefree.
    \end{thm}
   
\begin{proof}
Let $n$ be a natural number, $p,q,r$ primes with $r\mid \mid n$ and $\gcd(n,pq)=1$.  Set $n=n_0r.$

    Let \( a_1 = 1 < a_2 < \cdots < a_k \) be the divisors of \( n_0 \). 
Let \( B_0 \) be the adjacency matrix of \( D_{n_0} \). Let \( C \) be the matrix defined by 
\[
C_{ij} = 1 \iff a_i \mid a_j.\]
List the divisors of \( n_0 r p q \) in the order:
\begin{align*}
&\underbrace{a_1, \ldots, a_k}_{(1)}, \quad 
\underbrace{r a_1, \ldots, r a_k}_{(2)}, \quad 
\underbrace{p a_1, \ldots, p a_k}_{(3)}, \quad 
\underbrace{p r a_1, \ldots, p r a_k}_{(4)}, \\
&\underbrace{q a_1, \ldots, q a_k}_{(5)}, \quad 
\underbrace{q r a_1, \ldots, q r a_k}_{(6)}, \quad 
\underbrace{p q a_1, \ldots, p q a_k}_{(7)}, \quad 
\underbrace{p q r a_1, \ldots, p q r a_k}_{(8)}.
\end{align*}
We note that \[
C_{ij} = 1 \iff a_i \mid a_j \quad \iff \quad r a_i \mid r a_j.
\]
Then the adjacency matrix for \( D_{n_0 r p q} \) is of the form:
\[
M = 
\begin{bmatrix}
B_0 & C & C & C & C & C & C & C \\
C^T & B_0 & 0 & C & 0 & C & 0 & C \\
C^T & 0 & B_0 & C & 0 & 0 & C & C \\
C^T & C^T & C^T & B_0 & 0 & 0 & 0 & C \\
C^T & 0 & 0 & 0 & B_0 & C & C & C \\
C^T & C^T & 0 & 0 & C^T & B_0 & 0 & C \\
C^T & 0 & C^T & 0 & C^T & 0 & B_0 & C \\
C^T & C^T & C^T & C^T & C^T & C^T & C^T & B_0 \\
\end{bmatrix}
= 
\begin{bmatrix}
B & A & A & A \\
A^T & B & 0 & A \\
A^T & 0 & B & A \\
A^T & A^T & A^T & B \\
\end{bmatrix},
\]
where
\[
B = \begin{bmatrix}
B_0 & C \\
C^T & B_0 \\
\end{bmatrix},
\quad
A = \begin{bmatrix}
C & C \\
0 & C \\
\end{bmatrix}.
\]
We note that $B$ is precisely the adjacency matrix for $D_n$. We have
\[
\begin{aligned}
| M - \lambda I_N| &= \begin{vmatrix}
B_0-\lambda I & C & C & C & C & C & C & C \\
C^T & B_0-\lambda I & 0 & C & 0 & C & 0 & C \\
C^T & 0 & B_0-\lambda I & C & 0 & 0 & C & C \\
C^T & C^T & C^T & B_0-\lambda I & 0 & 0 & 0 & C \\
C^T & 0 & 0 & 0 & B_0-\lambda I & C & C & C \\
C^T & C^T & 0 & 0 & C^T & B_0-\lambda I & 0 & C \\
C^T & 0 & C^T & 0 & C^T & 0 & B_0-\lambda I & C \\
C^T & C^T & C^T & C^T & C^T & C^T & C^T & B_0-\lambda I \\
\end{vmatrix}\\
&=
\pm\begin{vmatrix}
B_0 - \lambda I & C & C & C & C & C & C & C \\
0 & B_0 - \lambda I & -B_0 + \lambda I & 0 & 0 & C & -C & 0 \\
C^T & 0 & B_0 - \lambda I & C & 0 & 0 & C & C \\
C^T & C^T & C^T & B_0 - \lambda I & 0 & 0 & 0 & C \\
0 & 0 & -B_0 + \lambda I & -C & B_0 - \lambda I & C & 0 & 0 \\
0 & 0 & -C^T & -B_0 + \lambda I & C^T & B_0 - \lambda I & 0 & 0 \\
0 & C^T & -C^T & 0 & 0 & B_0 - \lambda I & -B_0 + \lambda I & 0 \\
C^T & C^T & C^T & C^T & C^T & C^T & C^T & B_0 - \lambda I
\end{vmatrix}.
\end{aligned}
\]
Here, $I=I_k$ and we the following (block) row operations:
\begin{align*}
(2):\quad & R_2 \leftarrow R_2 - R_3 \\
(5):\quad & R_5 \leftarrow R_5 - R_3  \\
(6):\quad & R_6 \leftarrow R_6 - R_4 \\
(7):\quad & R_7 \leftarrow R_6 - R_7.
\end{align*}
Thus
\[
\begin{aligned}
| M - \lambda I_N |   =
&\pm 
\begin{vmatrix}
I & 0 & 0 & 0 & 0 & 0 & 0 & 0 \\
0 & I & 0 & 0 & 0 & 0 & 0 & 0 \\
0 & 0 & I & 0 & 0 & 0 & 0 & 0 \\
0 & 0 & 0 & I & 0 & 0 & 0 & 0 \\
0 & 0 & 0 & 0 & B_0 - \lambda I & C &0&0\\
0 & 0 & 0 & 0 & C^T & B_0 - \lambda I & 0 & 0 \\
0 & 0 & 0 & 0 & 0 & 0 & I & 0 \\
0 & 0 & 0 & 0 & 0 & 0 & 0 & I \\
\end{vmatrix}\times\\
&\begin{vmatrix}
B_0 - \lambda I & C & C & C & C & C & C & C \\
0 & B_0 - \lambda I & -B_0 + \lambda I & 0 & 0 & C & -C & 0 \\
C^T & 0 & B_0 - \lambda I & C & 0 & 0 & C & C \\
C^T & C^T & C^T & B_0 - \lambda I & 0 & 0 & 0 & C \\
0 & 0 & -I & 0 & I & 0 & 0 & 0 \\
0 & 0 & 0 & -I  & 0 &  I & 0 & 0 \\
0 & C^T & -C^T & 0 & 0 & B_0 - \lambda I & -B_0 + \lambda I & 0 \\
C^T & C^T & C^T & C^T & C^T & C^T & C^T & B_0 - \lambda I
\end{vmatrix}.
\end{aligned}
\]
In the above product of determinants, the first determinant is $f_n$. For the second determinant, which we denote temporarily as $S$, we move the block rows $R2$ and $R7$ to the top, we obtain
\[
\begin{aligned}
S&=\pm \begin{vmatrix}
0 & B_0 - \lambda I & -B_0 + \lambda I & 0 & 0 & C & -C & 0 \\
0 & C^T & -C^T & 0 & 0 & B_0 - \lambda I & -B_0 + \lambda I & 0 \\
B_0 - \lambda I & C & C & C & C & C & C & C \\
C^T & 0 & B_0 - \lambda I & C & 0 & 0 & C & C \\
C^T & C^T & C^T & B_0 - \lambda I & 0 & 0 & 0 & C \\
0 & 0 & -I & 0 & I & 0 & 0 & 0 \\
0 & 0 & 0 & -I  & 0 &  I & 0 & 0 \\
C^T & C^T & C^T & C^T & C^T & C^T & C^T & B_0 - \lambda I
\end{vmatrix} 
\end{aligned}
\]
Now in the above determinant, we move the block columns $C2$, $C_6$, $C_3$ and $C_7$ (in this order) to the left, we imply that

\[
\begin{aligned}
    S&=\pm\begin{vmatrix}
        B_0 - \lambda I & C & -B_0 + \lambda I & -C & 0 & 0 & 0 & 0 \\
C^T & B_0 - \lambda I & -C^T & -B_{0}+ \lambda I & 0 & 0 & 0 & 0 \\
C & C & C & C & B_0 - \lambda I & C & C & C \\
0 & 0 & B_0 - \lambda I & C & C^T & C & 0 & C \\
C^T & 0 & C^T & 0  &C^T& B_0 - \lambda I  & 0 & C \\
0 & 0 & -I & 0 & 0 & 0 & I & 0 \\
0 & I & 0 & 0 & 0 & -I & 0 & 0 \\
C^T & C^T & C^T & C^T & C^T & C^T & C^T & B_0 - \lambda I
    \end{vmatrix}\\
&=\pm 
\begin{vmatrix}
    B_0 - \lambda I & C & 0 & 0 & 0 & 0 & 0 & 0 \\
C^T & B_0 - \lambda I & 0 & 0 & 0 & 0 & 0 & 0 \\
0 & 0 & I & 0 & 0 & 0 & 0 & 0 \\
0 & 0 & 0 & I & 0 & 0 & 0 & 0 \\
0 & 0 & 0 & 0 & I & 0 & 0 & 0 \\
0 & 0 & 0 & 0 & 0 & I & 0 & 0 \\
0 & 0 & 0 & 0 & 0 & 0 & I & 0 \\
0 & 0 & 0 & 0 & 0 & 0 & 0 & I \\
\end{vmatrix}\times
\begin{vmatrix}
I & 0 & -I & 0 & 0 & 0 & 0 & 0 \\
0 & I & 0 & -I & 0 & 0 & 0 & 0 \\
C & C & C  &C &B_0 - \lambda I & C & C & C  \\
0 & 0 & B_0 \lambda I & C& C^T & C & 0 & C  \\
C^T & 0 & C^T &0 & C^T & B_0 - \lambda I & 0 & C\\
0 & 0 & -I & 0 & 0 & 0 & I & 0 \\
0 & I & 0 & 0 & 0& -I & 0 & 0  \\
C^T & C^T & C^T & C^T & C^T & C^T & C^T & C^T
\end{vmatrix}\\
&=\pm f_n \times (\text{a polynomial in $\lambda$}).
\end{aligned}
\]
Hence $f_{npq}$ is divisible by $f_n^2$.\\
Of course if $n>1$ is squarefree, we can choose for $r$ any prime dividing $n$.
\end{proof}

We discuss a different proof for \cref{thm:main}. In essence, the two proofs are the same. However, this approach might be applied to other problems such as bounding the multiplicity of certain eigenvalues (see for example \cref{cor:inclusion}). We first discuss the general idea that leads to this approach. Let $(S, \leq) $ be a partially ordered set and $D_{S}$ the compatibility graph associated to $S.$ Namely, the vertex set of $D_S$ is $S$ and two vertices $a, b$ are adjacent if $a<b$ or $b<a.$ As explained in \cref{subsec:equivalent}, $D_n$ is a special case of compatibility graphs. By definition, an eigenvector for $D_S$ associated with an eigenvalue $\lambda$ is the same as a function $f: S \to \C$ such that for each $\alpha \in S$ 
\[ \sum_{\beta < \alpha} f(\beta) + \sum_{\beta > \alpha} f(\beta) = \lambda f(\alpha). \]

For two partially ordered sets $S_1, S_2$, their cartesian product $S = S_1 \times S_2$ is defined in a canonical way. Specifically, as a set $S$ is the cartesian product of $S_1$ and $S_2$. In addition, we say that $(\beta_1, \beta_2) \leq (\alpha_1, \alpha_2)$ if and only if $\beta_1 \leq \alpha_1$ and $\beta_2 \leq \alpha_2.$ With this perspective, we can see that the partially ordered set $\Div(npq)$ is isomorphic to the product of $\Div(n)$ and $S_0$ where 
\[ S_0 =  \{(\alpha_1, \alpha_2) \in \Z^{2} \mid 0 \leq \alpha_i \leq 1 \} .\] 

The following theorem is a natural generalization of \cref{thm:main}.
\begin{thm} \label{thm:second_proof}
    Let $S$ be a partially ordered set and $\widehat{S}= S \times S_0.$ Let $f_{S}$ (respectively $f_{\widehat{S}})$ be the characteristic polynomial of $D_{S}$ (respectively $D_{\hat{S}})$. Then $f_{S} \mid f_{\widehat{S}}.$
\end{thm}
To prove this statement, our idea is to show that each eigenvector of $D_{S}$ lifts to an eigenvector for $D_{\widehat{S}}$ with the same eigenvalue. Because we are dealing with direct product, a natural approach would be using the tensor product. Let us recall its formal definition. 
\begin{definition} \label{def:tensor}
Let $S=S_1 \times S_2$ and $g: S_1 \to \C$ and $h:S_2 \to \C$ be two functions. The tensor product $f:=g \otimes h$ of $g$ and $h$ is the function $f: S_1 \times S_2 \to \C$ defined by the rule $f((s_1, s_2))= g(s_1)h(s_2).$
\end{definition}
It is not true in general that if $g, h$ are eigenvectors for $S_1, S_2$ then $g \otimes h$ is an eigenvector for $S.$ However, there are cases where it is indeed the case. In fact, $D_{S_0}$ has a special eigenvector that makes this work. More specifically, let $h: S_0 \to \C$ defined by the rule 
\[ h((0,0))=h((1,1))=0, h((0,1))=1, h((1,0))=-1.\]
Then $h$ has the property that for each $\gamma \in S_0$ 
\[ \sum_{\gamma<\epsilon} h(\epsilon) = \sum_{\gamma>\epsilon} h(\epsilon)=0. \]

In particular, $h$ is an eigenvector associated with the eigenvalue $0.$ We have the following observation.

\begin{prop} \label{prop:second_proof}

Let $S$ be a partially ordered set and $\widehat{S} = S \times S_0.$ Let $g: S \to \C$ be an eigenvector of $D_S$ associated with the eigenvalue $\lambda.$ Then $\widehat{g}:=g \otimes h: \widehat{S} \to \C$ is an eigenvector for $D_{\widehat{S}}$ associated with $\lambda.$

\end{prop}

\begin{proof}
Let $\widehat{\alpha} = (\alpha, \gamma)$ where $\alpha \in S$ and $\gamma \in S_0.$ We need to show that 
\[ \sum_{\widehat{\beta} < \widehat{\alpha}} \widehat{g}(\widehat{\beta}) + \sum_{\widehat{\beta} > \widehat{\alpha}} \widehat{g}(\widehat{\beta})  = \lambda \widehat{g}(\widehat{\alpha}).\]
For $\widehat{\beta}  = (\beta, \epsilon)$, we have $\widehat{\beta} < \widehat{\alpha}$ if and only if exactly one of the following conditions holds  

\begin{enumerate}
    \item $\alpha=\beta$ and $\epsilon < \gamma$. 
    \item $\epsilon = \gamma$ and $\beta< \alpha$.
    \item $\beta< \alpha$ and $\epsilon < \gamma.$
\end{enumerate}

Consequently, we can rewrite 

\begin{align*}
\sum_{\widehat{\beta} < \widehat{\alpha}} \widehat{g}(\widehat{\beta}) &= \sum_{\epsilon < \gamma} \widehat{g}(\alpha, \epsilon) + \sum_{\beta < \alpha} \widehat{g}(\alpha, \gamma) + \sum_{\epsilon < \gamma, \beta<\alpha} \widehat{g}(\beta, \epsilon) \\ 
& = g(\alpha) \sum_{\epsilon < \gamma} h(\epsilon) + h(\gamma) \sum_{\beta< \alpha} g(\beta) + \left(\sum_{\epsilon<\gamma} h(\epsilon) \right) \left(\sum_{\beta<\alpha} g(\beta) \right) \\
& =  h(\gamma) \sum_{ \beta < \gamma} g(\beta).
\end{align*}
The last equality follows from the fact that $\sum_{\epsilon <\gamma}h(\epsilon)=0.$ Similarly, we have 
\[ \sum_{\widehat{\alpha}<\widehat{\beta} } \widehat{g}(\widehat{\beta})  = h(\gamma) \sum_{\alpha< \beta} g(\beta).\]
Summing up and using the fact that $g$ is an eigenvector associated with $\lambda$ we conclude that 
\[ \sum_{\widehat{\beta} < \widehat{\alpha}} \widehat{g}(\widehat{\beta}) + \sum_{\widehat{\beta} > \widehat{\alpha}} \widehat{g}(\widehat{\beta}) = h(\gamma) \left( \sum_{ \beta < \gamma} g(\beta) +\sum_{ \alpha < \beta} g(\beta) \right) =  \lambda h(\gamma) g(\alpha) = \lambda \widehat{g}(\widehat{\alpha}).\] 
We conclude that $\widehat{g}$ is an eigenvector of $D_{\widehat{S}}$ associated with $\lambda.$
\end{proof}

We can now provide a proof of \cref{thm:second_proof}. 
\begin{proof}
Since $D_{S}$ is undirected, its adjacency matrix is diagonalizable. Let $\{g_1, g_2, \ldots, g_m\}$ be a complete system of eigenvectors associated with the system of eigenvalues $\{\lambda_1, \lambda_2,  \ldots, \lambda_m\}$ (here $m=|S|$). For each $1 \leq i \leq m$, let $\widehat{g}_i = g_i \otimes h $ as discussed in \cref{prop:second_proof}. Then $\widehat{g}_i$ is an eigenvector for $D_{\widehat{S}}$ associated with the eigenvalue $\lambda_i.$ Furthermore, by a standard property of the tensor product, the system $\{\widehat{g}_i\}_{i=1}^m$ is linearly independent. This shows that $f_{S} \mid f_{\widehat{S}}.$
\end{proof}

\begin{rem}
The analogues of Theorems 3.1 and 3.2 for partially ordered sets hold (one
proof for 3.1 given above). \ Specifically, let $(S,\leq )$ be a partially
ordered set, and write $S^{\prime }=S\times \{0,1\}$, which we order by $%
(s,a)\leq (t,b)\leftrightarrow s\leq t\ \&~a\leq b$, and use $f_{S}$ for the
characteristic polynomial of the compatibility graph corresponding to $S.$
Then the analogue of 3.1 is $f_{S}~|~f_{S^{\prime \prime }}$ and that of 3.2
is $f_{S^{\prime }}^{2}~|~f_{S^{\prime \prime \prime }}$. \ Both of these
can be proved by essentially the same arguments as those presented in
Theorems 3.1 and 3.2.
\end{rem}
\subsection{Special eigenvalues}
In this section, we discuss some special eigenvalues for $D_n.$ While the existence of these eigenvalues typically follow from \cref{thm:main}, and we may omit proofs below where this is clearly the case, we discuss some explicit eigenvectors which might be of independent interest. We also include some  numerical data for the multiplicity of these eigenvalues which seem to follow some interesting patterns.

\subsubsection{Eigenvalue $-2$}
Let $\lambda$ be a real number. Then $\lambda$ is an eigenvalue of $D_n$ if there is a vector $(v_d)_{d|n}$ such that 
\[ \sum_{(d,m) \in E(D_n)} v_m = \lambda v_d. \]
The case $\lambda=-2$ is special because we can rewrite this as 
\begin{equation} \label{eq:moebius}
\sum_{m|d} v_m + \sum_{d|m|n} v_m =0.
\end{equation}
We will use this fact to construct an eigenvector associated with $\lambda =-2$ under some conditions on $n.$

Let $\mu$ be the Moebius function. This function has the property that for each positive integer $m$
   \[
\sum_{d|m} \mu(m) =
\begin{cases} 
    1 & \text{if } m=1 \\
    0 & \text{else. }
\end{cases}
\]

\begin{prop}
    Suppose that $\mu(n)=-1$. Let $v=(v_d)_{d|n}$ be the vector defined by 
   \[
v_d =
\begin{cases} 
    0 & \text{if } d \in \{1, n\} \\
    \mu(d) & \text{else. }
\end{cases}
\]
Then $v$ is an eigenvector associated with the eigenvalue $\lambda =-2.$
    
\end{prop}
\begin{proof}
    We verify that \cref{eq:moebius} holds for all $d|n.$ If $d \not \in \{1, n\}$ then we have 
\[ \sum_{m|d} v_m + \sum_{d|m|n} v_m = \sum_{m|d} \mu(m) - \mu(1) + \mu(d) \sum_{m \mid \frac{n}{d}} \mu(m) - \mu(n) =0.\] 

    If $d=1$ then \cref{eq:moebius} becomes 
    \[ v_1 + \sum_{m|n} v_m = \sum_{d|n} \mu(d) - \mu(1) - \mu(n)=0.\]
    Similarly, we can check that $\cref{eq:moebius}$ holds for $d=n.$
\end{proof}
Let $n$ be a squarefree number and $\omega(n)$ the distinct number of prime factors of $n.$ The following table records the multiplicity of $\lambda=-2$ as an eigenvalue for $D_n.$
\vspace{0.5em}

\begin{table}[h]
\begin{tabular}{|c|c|c|c|c|c|c|c|c|c|c|c|c|}
\hline
$\omega(n)$ &2 & 3 & 4 & 5 & 6 & 7 & 8 & 9 & 10 & 11 & 12 & 13 \\
\hline
$m_{-2}$ &0 & 2 & 0 & 10 & 0 & 42 & 0 & 170 & 0 & 682 & 0 & 2730 \\
\hline
\end{tabular}
\caption{Multiplicity of $-2$ as an eigenvalue for $D_n$}
\end{table}
\begin{rem}
It seems that $-2$ is an eigenvalue if and only $\mu(n)=(-1)^{\omega(n)}=-1.$ Furthermore, a quick look on the The On-Line Encyclopedia of Integer Sequences reveals that the values  
\[2, 10, 42, 170, 682, 2730 \]
are the first few terms of the sequence A020988 (see \url{https://oeis.org/A020988}). We note that these terms satisfy the linear recurrence $x_{n+1}=4x_n+2$.

\end{rem}
\subsubsection{Eigenvalue $-1$}
\begin{lem} \label{lem:clique_1}
    Let $X$ be a clique set in a connected graph $G$. Suppose further that both $X$ and $G \setminus X$ are homogeneous sets in $G.$ Then $-1$ is an eigenvalue of $G$ with multiplicity at least $|X|-1.$
\end{lem}
\begin{proof}
By definition, $G$ is a joined union of the induced graph on $X$ and $G \setminus X.$  The conclusion then follows from the proof of \cite[Proposition 3.1]{join_graphs}. More specifically, the adjacency matrix of $G$ is of the form 
\[ A= \begin{pmatrix}
A_{X}& \ones \\ \ones & A_{G \setminus X}
\end{pmatrix} .\] 
Let $v=(v_1, \ldots, v_{|X|})$ be a vector such that $\sum_{i=1}^{|X|}v_i=0.$ Then, $w = (v_1, \ldots, v_{|X|}, 0 , \ldots, 0)$ is an eigenvector for $A_G$ associated with the eigenvalue $\lambda=-1.$
\end{proof}

\begin{cor}
    $-1$ is an eigenvalue of $D_n$ for every $n \geq 2.$
\end{cor}

\begin{proof}
    $X=\{1, n\}$ sastifies the conditions of \cref{lem:clique_1}.
\end{proof}

The following table records the multiplicity of $\lambda=-1$ as an eigenvalue for $D_n$ where $n$ is a squarefree number. 
\begin{table}[h]

\begin{tabular}{|c|c|c|c|c|c|c|c|c|c|c|c|c|}
\hline
$\omega(n)$ &2 & 3 & 4 & 5 & 6 & 7 & 8 & 9 & 10 & 11 & 12 & 13 \\
\hline
$m_{-1}$ &1 & 3 & 4 & 10 & 15 & 35 & 56 & 126 & 212 & 462 & 814 & 1716 \\
\hline
\end{tabular}
\caption{Multiplicity of $-1$ as an eigenvalue for $D_n$}
\end{table}

 \subsubsection{Eigenvalue $1$}
 \begin{prop}
    Suppose that $\mu(n)=-1$. Then $1$ is an eigenvalue of $D_n.$
 \end{prop}

 The following table records the multiplicity of $\lambda=1$ as an eigenvalue for $D_n$ where $n$ is a squarefree number. 

 \begin{table}[h]
\centering
\begin{tabular}{|c|c|c|c|c|c|c|c|c|c|c|c|c|}
\hline
$\omega(n)$ & 2 & 3 & 4 & 5 & 6 & 7 & 8 & 9 & 10 & 11 & 12 & 13 \\
\hline
$m_{1}$ &0 & 2 & 0 & 5 & 0 & 14 & 0 & 42 & 0 & 132 & 0 & 429 \\
\hline
\end{tabular}
\caption{Multiplicity of $1$ as an eigenvalue for $D_n$}
\end{table}

\subsubsection{Eigenvalue $0$}
\begin{prop}
    Suppose that $\mu(n)=1$. Then $0$ is an eigenvalue of $D_n.$
 \end{prop}
 The following table records the multiplicity of $\lambda=0$ as an eigenvalue for $D_n$ where $n$ is a squarefree number. 
 \begin{table}[h]
\begin{tabular}{|c|c|c|c|c|c|c|c|c|c|c|c|c|}
\hline
$\omega(n)$ &2 & 3 & 4 & 5 & 6 & 7 & 8 & 9 & 10 & 11 & 12 & 13 \\
\hline
$m_{0}$ &1 & 0 & 2 & 0 & 5 & 0 & 14 & 0 & 42 & 0 & 132 & 0 \\
\hline
\end{tabular}
\caption{Multiplicity of $0$ as an eigenvalue for $D_n$}
 \end{table}
\begin{rem}
It seems that $0$ (respectively $1$) is an eigenvalue if and only $\mu(n)=1$ (respectively $\mu(n)=-1)$. Furthermore, a quick look on the The On-Line Encyclopedia of Integer Sequences reveals that the values  $2, 5, 14, 42, 132, 429$ are the first few terms of the Catalan sequence A000108 (see \url{https://oeis.org/A000108}). This observation has also been discovered independently in \cite{Catalan}. 

\end{rem}

We provide below an partial explanation for this phenomenon. Let $n$ be a squarefree number such that $\omega(n)=2m.$ We recall from the discussion before the proof of \cref{prop:second_proof} that an eigenvector for $D_n$ can be considered as a function on the powerset $\{0,1\}^{2m} \to \C$ such that for all $\alpha \in \{0,1\}^{2m}$
\begin{equation*} \sum_{\beta<\alpha} f(\beta) + \sum_{\beta>\alpha} f(\beta) = \lambda f(\alpha). 
\end{equation*}
When $\alpha=0$, this condition becomes 
\[ \sum_{\beta<\alpha} f(\beta) + \sum_{\beta>\alpha} f(\beta) = 0. \]
Let $V_{m}$ be the space of all functions $f: \{0,1\}^{2m} \to \C$ such that for all $\alpha \in \{0,1\}^{2m}$
\[ \sum_{\beta<\alpha} f(\beta) = \sum_{\beta>\alpha} f(\beta)=0. \]
By the above discussion, $V_m$ is a subspace of the nullspace of the adjacency matrix for $D_n.$ We have the following observation whose proof is identical to the one given in \cref{prop:second_proof}.
\begin{lem}
Let $2m = 2m_1 + 2m_2$. For $f_1 \in V_{m_1}$ and $f_2 \in V_{m_2}$, let $f$ be the tensor product of $f_{1}$ and $f_2$ as defined in \cref{def:tensor}. Then $f \in V_{m}.$ 
\end{lem}

\begin{cor} \label{cor:inclusion}
    For each $(m_1, m_2)$ such that $2m= 2m_1 + 2m_2$, we have $V_{m_1} \otimes V_{m_2} \subset V_{m}.$
\end{cor}
We remark while \cref{cor:inclusion} sheds some insights, it is not sufficient to explain the dimension growth of the nullspace for $D_n.$ 

        
\subsection{The case where $n$ has at most two prime factors}

If $n$ is a prime power then $D_n$ is a complete graph. In this case, the spectrum of $D_n$ is well-known. In this section, we consider the next simplest case: namely $n$ is a product of two prime numbers. First, we consider the case where $n=pq^a$ where $p,q$ are distinct primes. 
\begin{prop} 
Let $M_{a}$ be an adjacency matrix for $D_{pq^{a}}$ where $p,q$ are distinct
primes, $a\geq 0.$   \ Then $\det (M_{a})=\det (M_{a+6})$.
\end{prop}

\begin{proof}
We note that $\det (M_{5})=1,$ and explicitly that, with respect to the
vertex ordering  $1,q,\cdots ,q^{5},p,pq,\cdots ,pq^{5}$ 
\[
M_{5}=\left[ 
\begin{array}{cccccccccccc}
0 & 1 & 1 & 1 & 1 & 1 & 1 & 1 & 1 & 1 & 1 & 1 \\ 
1 & 0 & 1 & 1 & 1 & 1 & 0 & 1 & 1 & 1 & 1 & 1 \\ 
1 & 1 & 0 & 1 & 1 & 1 & 0 & 0 & 1 & 1 & 1 & 1 \\ 
1 & 1 & 1 & 0 & 1 & 1 & 0 & 0 & 0 & 1 & 1 & 1 \\ 
1 & 1 & 1 & 1 & 0 & 1 & 0 & 0 & 0 & 0 & 1 & 1 \\ 
1 & 1 & 1 & 1 & 1 & 0 & 0 & 0 & 0 & 0 & 0 & 1 \\ 
1 & 0 & 0 & 0 & 0 & 0 & 0 & 1 & 1 & 1 & 1 & 1 \\ 
1 & 1 & 0 & 0 & 0 & 0 & 1 & 0 & 1 & 1 & 1 & 1 \\ 
1 & 1 & 1 & 0 & 0 & 0 & 1 & 1 & 0 & 1 & 1 & 1 \\ 
1 & 1 & 1 & 1 & 0 & 0 & 1 & 1 & 1 & 0 & 1 & 1 \\ 
1 & 1 & 1 & 1 & 1 & 0 & 1 & 1 & 1 & 1 & 0 & 1 \\ 
1 & 1 & 1 & 1 & 1 & 1 & 1 & 1 & 1 & 1 & 1 & 0%
\end{array}%
\right] 
\]

and  
\[
M_{5}^{-1}=\left[ 
\begin{array}{cccccccccccc}
-2 & 1 & 2 & 1 & -1 & -2 & -2 & -1 & 1 & 2 & 1 & -1 \\ 
1 & -2 & -1 & 0 & 1 & 1 & 2 & 0 & -1 & -1 & 0 & 1 \\ 
2 & -1 & -4 & -1 & 2 & 3 & 3 & 2 & -2 & -3 & -1 & 2 \\ 
1 & 0 & -1 & -2 & 1 & 2 & 1 & 1 & 0 & -2 & -1 & 1 \\ 
-1 & 1 & 2 & 1 & -2 & -1 & -2 & -1 & 1 & 2 & 0 & -1 \\ 
-2 & 1 & 3 & 2 & -1 & -4 & -3 & -2 & 1 & 3 & 2 & -2 \\ 
-2 & 2 & 3 & 1 & -2 & -3 & -4 & -1 & 2 & 3 & 1 & -2 \\ 
-1 & 0 & 2 & 1 & -1 & -2 & -1 & -2 & 1 & 2 & 1 & -1 \\ 
1 & -1 & -2 & 0 & 1 & 1 & 2 & 1 & -2 & -1 & 0 & 1 \\ 
2 & -1 & -3 & -2 & 2 & 3 & 3 & 2 & -1 & -4 & -1 & 2 \\ 
1 & 0 & -1 & -1 & 0 & 2 & 1 & 1 & 0 & -1 & -2 & 1 \\ 
-1 & 1 & 2 & 1 & -1 & -2 & -2 & -1 & 1 & 2 & 1 & -2%
\end{array}%
\right] 
\]%
Order the vertices of $M_{a}$ as $1,q,\cdots ,q^{a},p,pq,\cdots ,pq^{a}$,
and the vertices of $M_{a+6}$ as \newline $1,q,\cdots ,q^{a},p,pq,\cdots
,pq^{a},q^{a+1},\cdots ,q^{a+6},pq^{a+1},\cdots pq^{a+6}$ . Now if we write $%
\mathbf{1}_{r,s}$ for an $r\times s$ matrix of $1$'s and $\mathbf{0}_{r,s}$
for an $r\times s$ matrix of $0$'s, we have 
\[
M_{a+6}=\left[ 
\begin{array}{cc}
A & B \\ 
C & D%
\end{array}%
\right] 
\]%
where $A=M_{a},B=$ $\left[ 
\begin{array}{cc}
\mathbf{1}_{a+1,6} & \mathbf{1}_{a+1,6} \\ 
\mathbf{0}_{a+1,6} & \mathbf{1}_{a+1,6}%
\end{array}%
\right] ,C=B^{T},$ and $D=M_{5}$. \ Since $D$ is invertible, 
\[
\det (M_{a+6})=\det (D)\det (A-BD^{-1}C).
\]%
But $\det (D)=1$, and using the explicit computation of $D^{-1}=M_{5}^{-1}$
above, $BD^{-1}C=\mathbf{0}_{2a+2,2a+2}$. \ (We can check this by verifying $%
\mathbf{1}_{1,12}D^{-1}\mathbf{1}_{12,1}=[\mathbf{0}_{1,6}\ \mathbf{1}%
_{1,6}]D^{-1}\mathbf{1}_{12,1}=\mathbf{1}_{1,12}D^{-1}\left[ 
\begin{array}{c}
\mathbf{0}_{6,1} \\ 
\mathbf{1}_{6,1}%
\end{array}%
\right] =[\mathbf{0}_{1,6}\ \mathbf{1}_{1,6}]D^{-1}\left[ 
\begin{array}{c}
\mathbf{0}_{6,1} \\ 
\mathbf{1}_{6,1}%
\end{array}%
\right] =0.$)\ \ Therefore $\det $($M_{a+6})=\det (A)=\det (M_{a})$.
\end{proof}

\begin{prop} 
For distinct primes $p,q$ and $a\geq 0$, $D_{pq^{a}}$ has eigenvalue $0$ if
and only if $a\equiv 1\mod{6}$.
\end{prop}

\begin{proof}
After explicitly computing $\det (M_{a})=-1,0,3,5,4,1$ for $a=0,1,2,3,4,5$
respectively, this follows immediately from the previous proposition.
\end{proof}
The more general case is somewhat more involved. We can only prove the if direction. 

\begin{prop}
If $n=p^a q^b$ and $a \equiv b \equiv 1 \pmod{6}$ then $0$ is an eigenvalue of $D_n.$
\end{prop}
\begin{proof} Suppose that     \[
n = p^u q^v \quad \text{with} \quad u \equiv v \equiv 1 \pmod{6}.
\]
List the divisors of \( n \) in the order:
\[
1, q, q^2, \ldots, q^v, p, pq, pq^2, \ldots, pq^v, \ldots, p^u, p^u q, \ldots, p^u q^v.
\]
The adjacency matrix of \( D_n \) is of the form:
\[
M = 
\begin{bmatrix}
V & U & U & \cdots & U \\
U^T & V & U & \cdots & U \\
U^T & U^T & V & \cdots & U \\
\vdots & \vdots & \vdots & \ddots & \vdots \\
U^T & U^T & U^T & \cdots & V
\end{bmatrix},
\]
where
\[
V = 
\begin{bmatrix}
0 & 1 & 1 & \cdots & 1 \\
1 & 0 & 1 & \cdots & 1 \\
1 & 1 & 0 & \cdots & 1 \\
\vdots & \vdots & \vdots & \ddots & \vdots \\
1 & 1 & 1 & \cdots & 0 
\end{bmatrix}_{(v+1)\times(v+1)}
\quad \text{and} \quad
U = 
\begin{bmatrix}
1 & 1 & \cdots & 1 \\
0 & 1 & \cdots & 1 \\
\vdots & \vdots & \ddots & \vdots \\
0 & 0 & \cdots & 1
\end{bmatrix}_{(v+1)\times(v+1)}
\]
\setcounter{MaxMatrixCols}{20}
Consider the following column matrices of size \((v+1) \times 1\):
\[
\begin{aligned}
    A&=\left[\begin{array}{*{15}c}
            0& 1&1& 0 &-1&-1  &\cdots &0 & 1&1& 0&-1&-1 & 0 & 1 \\
           \end{array}\right]^T\\
    A'&=\left[\begin{array}{*{15}c}
           -1& 0& 1&1& 0 &-1  &\cdots &-1&0 & 1&1& 0&-1&-1 & 0  \\
           \end{array}\right]^T.
\end{aligned}
\]
(In $A$, the block $\begin{bmatrix}
     0& 1&1& 0 &-1&-1 
\end{bmatrix}$ is repeated $(v-1)/6$ times. The matrix $A'$ is obtained from $A$ by left-shifting.) Similary, we consider following column matrices of size \((v+1) \times 1\):
\[
\begin{aligned}
       B&=\left[\begin{array}{*{15}c}
           -1& -1& 0& 1 &1 &0 &\cdots &-1& -1& 0& 1 &1 &0 & -1 & -1  \\
           \end{array}\right]^T,\\
        B'&=\left[\begin{array}{*{15}c}
            0& -1&-1& 0 &1&1  &\cdots &0 & -1&-1& 0&1&1 & 0 & -1 \\
           \end{array}\right]^T\,\\
\end{aligned}
\]
and
\[
\begin{aligned}
       C&=\left[\begin{array}{*{15}c}
           1& 0& -1& -1 &0 &1 &\cdots &1& 0& -1& -1 &0 &1 &1 &0  \\
           \end{array}\right]^T,\\
        C'&=\left[\begin{array}{*{15}c}
            1& 1&0& -1 &-1&0  &\cdots &1& 1&0& -1 &-1&0&1&1  \\
           \end{array}\right]^T.
\end{aligned}
\]
Let $X$ be the following column matrix of size $[(u+1)(v+1)]\times 1$:
\[
X=\begin{bmatrix}
A & A' & B & B' & C&C' &\cdots & A & A' & B & B' & C&C' & A & A'
\end{bmatrix}^T.
\]
(In $X$, the block $\begin{bmatrix}
   A & A' & B & B' & C&C' 
\end{bmatrix}$ is repeated $(u-1)/6$ times.)

We claim  $MX=0$. In particular, it shows that $0$ is an eigenvalue of $M$.

We need to show that for every non-negative integers $s,t$ with $s+t=u\equiv 1\pmod 6$, one has
\[
\begin{bmatrix}
    U^T & \cdots &U^T & V &U&\cdots & U
\end{bmatrix} X =0.
\]
Here in the first matrix,  $U^T$ appears $s$ times and $U$ appears $t$ times.

Note that $A+A'+B+B'+C+C'=0$. Hence $U^T(A+A'+B+B'+C+C')=0=U(B+B'+C+C'+A+A')$. We only need to consider the following six cases.

\medskip
\noindent {\bf Case 1}: $s\equiv 0\pmod 6$ and $t\equiv 1\pmod 6$. In this case, it suffices to show that \[\begin{bmatrix}
    V & U 
\end{bmatrix} \begin{bmatrix}
    A\\A'
\end{bmatrix}=0.\]

Let $\begin{bmatrix}
    V & U 
\end{bmatrix} \begin{bmatrix}
    A\\A'
\end{bmatrix}=\begin{bmatrix}s_1\\s_2\\\vdots\end{bmatrix}$. We have
\[
s_1=\sum_{i\geq 2}a_i+ \sum_{i\geq 1}a'_i= 1+(-1)=0.
\]
On the other hand, the difference between the $(i+1)$-st row and the $i$th row of $[V \mid U]$ is the row
\[
\left [\; 0\quad \cdots\quad 0 \quad 1 \quad -1 \quad 0\quad\cdots \quad 0\quad \mid\quad  0\quad\cdots \quad 0\quad -1\quad 0\quad\cdots \quad 0\;\right].
\]
Here, the first $-1$ is in the $i+1$-th position and the second $-1$ is in the $(s+1)+i$-th position. We deduce that
\[
s_{i+1}-s_i=a_i-a_{i+1}-a'_i=0.
\]
Hence $s_i=0$, for every $i$.

\medskip
\noindent {\bf Case 2}: $s\equiv 1\pmod 6$ and $t\equiv 0\pmod 6$. In this case, it suffices to show that \[\begin{bmatrix}
    U^T & V 
\end{bmatrix} \begin{bmatrix}
    A\\A'
\end{bmatrix}=0.\]

Let $\begin{bmatrix}
    U^T & V 
\end{bmatrix} \begin{bmatrix}
    A\\A'
\end{bmatrix}=\begin{bmatrix}s_1\\s_2\\\vdots\end{bmatrix}$. We have
\[
s_1=a_1+ \sum_{i\geq 2}a'_i= 0+0=0.
\]
On the other hand, the difference between the $(i+1)$-st row and the $i$th row of $[U^T \mid V]$ is the row
\[
\left [\; 0\quad \cdots\quad 0 \quad 1  \quad 0\quad\cdots \quad 0\quad \mid\quad  0\quad\cdots \quad 0 \quad  1\quad -1\quad 0\quad\cdots \quad 0\;\right].
\]
Here, the first $1$ is in the $i+1$-th position and the second $1$ is in the $(s+1)+i$-th position. We deduce that
\[
s_{i+1}-s_i=a_{i+1}+a'_i-a'_{i+1}=0.
\]
Hence $s_i=0$, for every $i$.

Using  similar arguments as in the previous two cases, we can show that
\[
VB+UB'= [\;s \quad s\quad\cdots\quad s\;]^T= U^TB+VB',
\]
where $s=\sum_{i\geq 2}b_i+ \sum_{i\geq 1}b'_i=-2=b_1+ \sum_{i\geq 2}b'_i$, and that
\[
\begin{aligned}
&VA'+UB=U^TA'+VB=[-2\quad -2\quad \cdots \quad -2\;]^T,\\
&VB'+UC= U^TB'+VC=0, \\
&VC+UC'=U^TC+VC'=[\;2\quad 2\quad \cdots \quad 2\;]^T.   
\end{aligned}
\]

\medskip
\noindent {\bf Case 3}: $s\equiv 2\pmod 6$ and $t\equiv 5\pmod 6$. In this case, it suffices to show that \[\begin{bmatrix}
    U^T & U^T & V & U & U & U&U&U 
\end{bmatrix} \begin{bmatrix}
    A\\A'\\B\\B'\\C\\C'\\A\\A'
\end{bmatrix}=0.\]

We have
\[
\begin{aligned}
   &U^T(A+A')+VB+U(B'+C+C'+A+A')=-VA'+U^TA' +VB-UB \\
   &= (U^TA'+VB)-(VA'+UB)=0.
\end{aligned}
\]

\medskip
\noindent {\bf Case 4}: $s\equiv 3\pmod 6$ and $t\equiv 4\pmod 6$. In this case, it suffices to show that \[\begin{bmatrix}
    U^T & U^T & U^T & V & U & U & U&U 
\end{bmatrix} \begin{bmatrix}
    A\\A'\\B\\B'\\C\\C'\\A\\A'
\end{bmatrix}=0.\]

We have 
\[
\begin{aligned}
&U^T(A+A'+B) + VB'+U(C+C'+A+A')\\
&=U^T(A+A')+VB+U(B'+C+C'+A+A') +U^TB+VB' -VB -UB' \\
&=U^TB+VB' -(VB +UB')=0.
\end{aligned}
\]

\medskip
\noindent {\bf Case 5}: $s\equiv 4\pmod 6$ and $t\equiv 3\pmod 6$. In this case, it suffices to show that \[\begin{bmatrix}
    U^T & U^T & U^T & U^T&V & U & U & U 
\end{bmatrix} \begin{bmatrix}
    A\\A'\\B\\B'\\C\\C'\\A\\A'
\end{bmatrix}=0.\]

We have 
\[
\begin{aligned}
&U^T(A+A'+B+B') + VC+U(C'+A+A')\\
&=U^T(A+A'+B) + VB'+U(C+C'+A+A') +U^TB'+VC -VB' -UC \\
&=U^TB'+VC  -(VB' +UC)=0.
\end{aligned}
\]

\medskip
\noindent {\bf Case 6}: $s\equiv 5\pmod 6$ and $t\equiv 2\pmod 6$. In this case, it suffices to show that \[\begin{bmatrix}
    U^T & U^T & U^T &U^T &U^T & V & U & U 
\end{bmatrix} \begin{bmatrix}
    A\\A'\\B\\B'\\C\\C'\\A\\A'
\end{bmatrix}=0.\]
We have 
\[
\begin{aligned}
&U^T(A+A'+B+B'+C) + VC'+U(A+A')\\
&=U^T(A+A'+B+B') + VC+U(C'+A+A')+U^TC+VC' - VC -UC' \\
&=U^TC+VC'  -(VC +UC')=0.
\end{aligned}
\]

\end{proof}

\section*{Acknowledgements}
This project was initiated during the workshop on Massey Products in Galois Cohomology, held from June 13 to June 16, 2024, at the University of Ottawa. We gratefully acknowledge the organizers-Professors Stefan Gille, Luisa Liboni and Kirill Zainoulline-for their kind invitation and hospitality. NDT gratefully acknowledges the Vietnam Institute for Advanced Study in Mathematics (VIASM) for hospitality and support during a visit in 2025.

 \end{document}